\documentclass[a4paper,10pt]{article}

\usepackage[english]{babel}
\usepackage[utf8]{inputenc} 
\usepackage[T1]{fontenc}

\usepackage{amsmath,amssymb,amsthm,mathrsfs,stmaryrd}

\usepackage[a4paper]{geometry}

\newtheorem{theorem}{Theorem}[section]
\newtheorem{lemma}[theorem]{Lemma}
\newtheorem{proposition}[theorem]{Proposition}

\newtheorem{remark}[theorem]{Remark}
\newtheorem{corollary}[theorem]{Corollary}

\makeatletter

\@addtoreset{equation}{section}
\makeatother
\usepackage{tikz}
\usepackage{enumitem}

\newcommand{\bzero}{\mathbf{0}}

\newcommand{\rc}{\mathrm{c}}

\newcommand{\rd}{\mathrm{d}}
\newcommand{\ud}{\,\mathrm{d}}

\renewcommand{\div}{\mathrm{div}}

\newcommand{\be}{\boldsymbol{e}}

\renewcommand{\bf}{\boldsymbol{f}}
\newcommand{\bF}{\boldsymbol{F}}

\newcommand{\Id}{\mathrm{Id}}
\newcommand{\Lip}{\mathrm{Lip}}
\newcommand{\loc}{\mathrm{loc}}
\newcommand{\bn}{\boldsymbol{n}}
\newcommand{\N}{\mathbf{N}}

\newcommand{\p}{\partial}

\newcommand{\vphi}{\varphi}

\newcommand{\R}{\mathbf{R}}

\renewcommand{\tilde}[1]{\widetilde{#1}}
\newcommand{\bu}{\boldsymbol{u}}
\newcommand{\uloc}{\mathrm{uloc}}

\newcommand{\rw}{\mathrm{w}}
\newcommand{\weak}{\mathrm{weak}}
\newcommand{\Z}{\boldsymbol{Z}}

\title{Well-posedness of the Stokes-transport system in bounded domains and in the infinite strip}

\author{Antoine Leblond\footnote{Sorbonne Université, CNRS, Université de Paris, Laboratoire Jacques-Louis Lions (LJLL), F-75005 Paris, France, mail: leblond@ljll.math.upmc.fr}}

\date{March, 2021}

\begin{document}

\maketitle

\begin{abstract}
	We consider the Stokes-transport system, a model for the evolution of an incompressible viscous fluid with inhomogeneous density. This equation was already known to be globally well-posed for any $L^1\cap L^\infty$ initial density with finite first moment in $\R^3$. We show that similar results hold on different domain types. We prove that the system is globally well-posed for $L^\infty$ initial data in bounded domains of $\R^2$ and $\R^3$ as well as in the infinite strip $\R\times(0,1)$. These results contrast with the ill-posedness of a similar problem, the incompressible porous medium equation, for which uniqueness is known to fail for such a density regularity.
\end{abstract}


% \tableofcontents

\section{Introduction}
	This contribution is dedicated to the study of the following coupling between the transport equation and the Stokes equation,
	\begin{equation} \label{eq:base} \tag{1.0}
	\left\{
	\begin{array}{rcl}
	\p_t\rho + \bu\cdot\nabla \rho &=&0 \\
	-\Delta \bu + \nabla p &=& -\rho \be_z \\
	\div \, \bu &=& 0 \\
	\rho|_{t=0} &=& \rho_0,
	\end{array}
	\right.
	\end{equation}
	where $\rho$, $\bu$ and $p$ respectively stand for the density, velocity and pressure of a fluid, $\rho_0$ is the initial density profile and $\be_z$ is the vertical upward unitary vector. This system is a model of evolution of an incompressible and viscous fluid having an inhomogeneous density or buoyancy subject to the gravity directed by $-\be_z$. It differs from the classical Boussinesq equation by neglecting the velocity self-advection term and the diffusion of the density. It is especially derived as the mesoscopic model of a cloud of inertialess particles sedimenting in a Stokes fluid, see for instance \cite{hofer,mecherbetsed}. This system belongs to a broad family of transport equations with non-local velocity field, the active scalar equations, meaning that $\bu$ depends on $\rho$ in a non-local way. The vorticity equation, the surface quasi-geostrophic equation and the incompressible porous medium equation (IPM) are examples of extensively studied systems having this structure, see \cite{bae} for an overview. In particular, the IPM is a well-known model for the evolution of an incompressible inhomogeneous fluid inside a porous material, which writes as the system (\ref{eq:base}) where one replaces the Stokes equation by Darcy's law, namely
	\begin{equation*}
	\bu + \nabla p = - \rho\be_z.
	\end{equation*}
	
	To our knowledge, (\ref{eq:base}) has been shown to be well-posed in the whole space $\R^3$ by Höfer \cite{hofer} and Mecherbet \cite{mecherbet}. In particular, Mecherbet proved that for any $\rho_0 \in L^1\cap L^\infty$ having a finite first moment, the system (\ref{eq:base}) admits a unique solution, global in time,
	\begin{equation*}
	(\rho,\bu) \in L^\infty(\R_+; L^1(\R^3)\cap L^\infty(\R^3)) \times L^\infty(\R_+;W^{1,\infty}(\R^3)).
	\end{equation*}
	We thereafter state the well-posedness of (\ref{eq:base}) for initial data $\rho_0$ in $L^\infty(\Omega)$, for any regular enough bounded subdomain $\Omega$ of $\R^2$ and $\R^3$, as well as in the infinite strip $\R\times(0,1)$. These results stand out with the ill-posedness of IPM, for which uniqueness fails on various domains for weak solutions associated to $L^\infty$ initial data, see \cite{cordoba2011lack,szekelyhidi2012relaxation}. Even for Sobolev data, the question of global in time well-posedness of IPM is still open, see \cite{kiselev2021small}, although the particular case of a small and smooth perturbation of a stationary density profile does lead to a global solution, as proved in \cite{castro2019global}.
	
	In both Höfer and Mecherbet works, the proof of existence of a solution relies on a fixed point argument, respectively a contracting map and a Picard iteration. The latter consists in solving successively the Stokes and the transport equations, providing function sequences that appear to be convergent. This convergence is proven thanks to a stability estimate on the solutions of the transport equation, combined with an energy estimate for the Stokes equation. In the full space $\R^3$, one can express the solution of the Stokes equation as the convolution with the adequate Green kernel. Thanks to this explicit formula, Mecherbet established the stability estimate on solutions of the Stokes equation, controlled by the Wasserstein distance between their source terms. Combined with a stability estimate for the transport equation, expressed with the Wasserstein distance, this provides the contraction inequality allowing to apply the Picard argument. In particular, the Wasserstein distance allows to state the stability estimate without any derivability assumption on the density. In the following, we will emphasize the fact that the shape of the spatial domain on which one solves (\ref{eq:base}) strongly conditions the way one have to address the Stokes part. Apart from that, the classical transport theory is not as much sensitive to the geometry of the domain, as long as the velocity field is bounded and Lipschitz, with extra condition on the possible boundaries.  Moreover it appears that $\rho$ is the push-forward of $\rho_0$ by the flow of $\bu$. In particular, if $\rho_0$ is a patch, the density $\rho$ will remain a patch at all times.
	
	In this paper, we follow a strategy similar to the one adopted by Mecherbet \cite{mecherbet} in order to prove the well-posedness of (\ref{eq:base}) on bounded domains of $\R^2$ and $\R^3$. Since the Green kernel associated to the Stokes equation is no longer explicit for such domains in general, the remaining exploitable tools are the energy estimate and an elliptic gain of regularity due to Stokes equation. We also replaced the Wasserstein distance, the use of which requires an explicit kernel, by some negative Sobolev norms. Notice that both metrics are relatable, see for instance \cite[subsection 5.5.2]{santambrogio}. We also impose the Dirichlet boundary condition on the velocity, considering more precisely the system
	\begin{equation} \label{eq:st}
		\left\{
		\begin{array}{rcll}
			\p_t \rho + \bu\cdot\nabla \rho &=& 0 &\text{in } \R_+\times\Omega, \\
			-\Delta \bu + \nabla p &=& -\rho \be_z &\text{in } \R_+\times \Omega, \\
			\div \, \bu &=& 0 &\text{in } \R_+ \times \Omega, \\
			\bu &=& \bzero &\text{in } \R_+ \times \p\Omega, \\
			\rho|_{t=0} &=& \rho_0 &\text{in } \Omega.
		\end{array}
		\right.
	\end{equation}
    Our first result is the following well-posedness theorem for this system, which comes together with a stability estimate left in Proposition \ref{prop:stabsys}.
	\begin{theorem} \label{thm:1}
		Let $\Omega$ be a bounded domain of class $C^2$ of $\R^d$ for $d=2$ or $3$. For any $\rho_0 \in L^\infty(\Omega)$ there exists a unique solution $(\rho,\bu)$ to (\ref{eq:st}) in
		\[L^\infty(\R_+;L^\infty(\Omega)) \times L^\infty(\R_+;W^{1,\infty}(\Omega)).\]
	\end{theorem}
	
	We further deal with the infinite strip domain $\R\times(0,1)$. The main difference lies in the analysis of the Stokes equation. To work with $L^\infty$ densities requires to deal with velocity profiles that are not square-integrable. To overcome this observation we work with Kato spaces, also called uniformly local Sobolev spaces. Using some classical tools we provide a proof of the well-posedness of the Stokes equation in this framework, since we could not find this result in the literature. In particular we prove that a $L^\infty$ datum induces a $W^{1,\infty}$ solution, as in the bounded domain case, which will provide handy Lipchitz velocity fields to deal further with the transport part. The method consists in considering an increasing covering of bounded open subsets of the strip, and to solve the Stokes problem on each of these subdomains. This provides a sequence of functions on $\R\times(0,1)$. By managing carefully the interior estimates of the elements of this sequence, we prove its boundedness in the uniformly locally $H^1$ function space, denoted $H^1_\uloc$ and defined in paragraph \ref{sbsbsec:spaces}. Therefore we obtain a $H^1_\uloc$ solution by a compact argument. Uniqueness comes from the well-posedness in the classical $H^1$ framework. Using the elliptic regularity of the Stokes equation in bounded domains, we prove the $H^2_\uloc$ regularity of the solution on the strip. Hence we obtain a leverage to establish existence of a unique solution in $W^{2,q}_\uloc$ for $L^q_\uloc$ data for exponents $2 < q < \infty$, and to conclude to the well-posedness of the problem in $W^{1,\infty}$ for $L^\infty$ data by Sobolev embeddings. The proof of the well-posedness of the coupling then lies in the extension of results used in the bounded domain case to the infinite strip through the uniformly local topology. Precisely, we consider in $\Omega = \R\times(0,1)$ the system
	\begin{equation} \label{eq:st0}
		\left\{
		\begin{array}{rcll}
			\p_t \rho + \bu\cdot\nabla \rho &=& 0 &\text{in } \R_+ \times \Omega, \\
			- \Delta \bu + \nabla p &=& -\rho \be_z &\text{in } \R_+\times \Omega, \\
			\div \, \bu &=& 0 &\text{in } \R_+\times\Omega, \\
			\bu &=& \bzero &\text{in } \R_+\times\p\Omega, \\
			\int u_1 \ud z &=& 0 & \text{in } \R_+\times\R, \\
			\rho|_{t=0} &=& \rho_0 &\text{in } \Omega,
		\end{array}
		\right.
	\end{equation} 
	where the extra condition on the flux $\int u_1 \ud z$ is introduced and discussed in paragraph \ref{sbsbsec:flux}. In the end, the well-posedness of (\ref{eq:st0}) almost writes as in Theorem \ref{thm:1}, with a similar stability result stated in Proposition \ref{prop:stabstrip}.
	\begin{theorem} \label{thm:2}
		Let $\Omega = \R\times(0,1)$. For any $\rho_0 \in L^\infty(\Omega)$ there exists a unique solution $(\rho,\bu)$ to (\ref{eq:st0}) in
		\[L^\infty(\R_+;L^\infty(\Omega)) \times L^\infty(\R_+;W^{1,\infty}(\Omega)).\]
	\end{theorem}
	It is rather natural to wonder if this result extends to the case of the layer domain $\R^2\times(0,1)$. It appears that the Stokes problem in this unbounded domain raises some additional difficulties, among those finding a functional space in which the problem admits a unique solution, in accordance with the handling of the interior estimates mentioned in the previous paragraph, see Remark \ref{rk:lay} for further details. We abstain from considering this case in the present work.
	
	The paper is organized in Section \ref{sec:bounded}, dedicated to the bounded domain case, and Section \ref{sec:unbounded}, dedicated to the infinite strip case. Both sections are ordered in the same way. Subsections 1 recall and prove the necessary prerequisites about the Stokes equation, including the well-posedness complete proof in the infinite strip in Section \ref{sec:unbounded}. Subsections 2 contain preliminary results concerning the transport theory, and in particular the stability estimates proofs. Subsections 3 are dedicated to the proofs of Theorem \ref{thm:1} and Theorem \ref{thm:2}, respectively. Ultimately, we state and prove the stability estimates of the coupling in Subsections 4.

	\subsection*{Definitions and notations}
	The dimension $d$ is always $2$ or $3$. A \emph{domain} $\Omega$ is a non-empty open and simply connected subset of $\R^d$. The Bochner spaces are denoted by $L^q([0,T);B(\Omega))$ with $1 \leq q \leq \infty$, $T \in [0,\infty]$ and $B(\Omega)$ a Banach or a Fréchet space of functions defined in $\Omega$. It is endowed with its classical norm denoted here $\|\cdot\|_{L^q(0,T;B)}$, the space domain being specified only when differing from the whole domain $\Omega$. A vector valued map is denoted by a bold symbol, implicitly of size $d$. We note $\bu|_{\p\Omega}$ the trace of $\bu$ on the boundary of a domain $\Omega$, when the boundary is regular enough to define it. We sometimes write $f \equiv c$ to signify that a function $f$ is constant to $c$ with respect to time. We denote by $C$ any non-negative constant that is adjusted from one line to another, independent of the data and we specify its eventual space or exponent dependencies when necessary. We sometimes write $f \lesssim g$, meaning there exists such a constant $C$ such that $f \leq C g$, as well as $f \simeq g$, meaning $f \lesssim g$ and $g \lesssim f$. 
	
\section{Well-posedness of the coupling in a bounded domain} \label{sec:bounded}

	In this section, and unless stated otherwise, $\Omega$ denotes a bounded domain of $\R^d$ with Lipschitz boundary. 
	
	\subsection{Preliminaries on the Stokes problem in a bounded domain}
		Let us recall the Stokes problem on $\Omega$, defined in the weak sense,
		\begin{equation} \label{eq:s}
		\left\{
		\begin{array}{rcll}
		- \Delta \bu + \nabla p &=& \bf &\text{in } \Omega, \\
		\div \, \bu &=& g &\text{in } \Omega, \\
		\bu &=& \bzero & \text{on } \p\Omega,
		\end{array}
		\right.
		\end{equation}
		where $\bf$ and $g$ belong to functional spaces specified further. Notice that it is necessary for $g$ to satisfy the following compatibility condition due to the homogeneous assumption, 
		\begin{equation} \label{eq:compcond}
			\int_\Omega g = \int_\Omega \div\,\bu = \int_{\p\Omega} \bu\cdot\bn = 0.
		\end{equation}
		The well-posedness of this problem is well known for Sobolev data, see in particular \cite[Theorem IV.6.1 and Exercise I.6.3]{galdi}, reported below. We especially use that $L^\infty$ data induce $W^{1,\infty}$ solutions, see the following corollary, which will provide velocity fields easy to deal with in the transport part.
		\begin{theorem}[Galdi] \label{thm:galdi}
			Let $\Omega$ be a bounded domain of $\R^d$ of class $C^2$ and let $1<q<\infty$.
			
			1. For any $\bf \in L^q(\Omega)$ and $g \in W^{1,q}(\Omega)$ satisfying (\ref{eq:compcond}), there exists a unique\footnote{Unless explicitly stated otherwise, the canonical representative of $p$ considered is the one having zero average over $\Omega$.} pair $(\bu,p)$ in $W^{2,q}(\Omega)\times \left(W^{1,q}(\Omega)/ \R\right)$ satisfying (\ref{eq:s}), which moreover obeys the inequality
			\begin{equation}\label{eq:sest}
				\|\bu\|_{W^{2,q}} + \|p\|_{W^{1,q}} \leq C\left( \|\bf\|_{L^q} + \|g\|_{W^{1,q}}\right);
			\end{equation}
			
			2. For any $\bf \in W^{-1,q}(\Omega)$ there exists a unique pair $(\bu,p)$ in $W^{1,q}(\Omega)\times \left(L^q(\Omega)/\R\right)$ satisfying (\ref{eq:s}) with $g = 0$, which moreover obeys the inequality
			\begin{equation} \label{eq:sestd}
				\|\bu\|_{W^{1,q}} + \|p\|_{L^q} \leq C\|\bf\|_{W^{-1,q}}.
			\end{equation}
		\end{theorem}
	The $L^\infty$ to $W^{1,\infty}$ regularity of the problem is deduced from the Sobolev embeddings of $W^{2,4}(\Omega)$ in $W^{1,\infty}(\Omega)$ and $W^{1,4}(\Omega)$ in $L^\infty(\Omega)$ for $d = 2$ and $3$.
		\begin{corollary}
			Let $\Omega$ be a bounded domain of $\R^d$ of class $C^2$. For any $\bf \in L^\infty(\Omega)$, there exists a unique pair $(\bu,p)$ in $W^{1,\infty}(\Omega) \times \left(L^\infty(\Omega)/\R\right)$ satisfying (\ref{eq:s}), which moreover obeys the inequality
			\begin{equation} \label{eq:sestub}
				\|\bu\|_{W^{1,\infty}} + \|p\|_{L^\infty} \leq C \|\bf\|_{L^\infty}.
			\end{equation}
		\end{corollary}
		
\subsection{Preliminaries on the transport equation and stability estimates}

	Let us consider the transport equation, in the weak sense, for a given vector field $\bu \in L^\infty(\R_+;W^{1,\infty}(\Omega))$ satisfying the Dirichlet condition $\bu|_{\p\Omega} \equiv 0$,
	\begin{equation} \label{eq:t}
		\left\{
		\begin{array}{rcll}
			\p_t \rho + \bu \cdot \nabla \rho &= & 0 &\text{in } \R_+\times\Omega,\\
			\rho(0,\cdot) &=& \rho_0 &\text{in } \Omega.
		\end{array}
		\right.
	\end{equation} 
	We recall the definition of the \emph{characteristic (map)} or \emph{flow} $X$ associated to the vector field $\bu$, as the solution of
	\begin{equation*}
	\forall s, t \in \R_+, \forall x \in \Omega, \quad
	\left\{
	\begin{array}{rcl}
	\p_t X(t,s,x) &= &\bu(t,X(t,s,x))  \\
	X(s,s,x) &= &x. 
	\end{array}
	\right.
	\end{equation*} 
	The Cauchy-Lipschitz theory ensures that $X$ is well defined, and that for any $s,t \in \R^+, X(t,s,\cdot)$ is a homeomorphism from $\Omega$ onto itself, satisfying the composition principle
	\[ \forall r, s,t \in \R_+, \quad X(t,s,\cdot)\circ X(s,r,\cdot) = X(t,r,\cdot). \]
	In particular we have the relation $X(t,0,\cdot)^{-1} = X(0,t,\cdot)$. From now on, we use indifferently the following notations
	\[ \forall t \in R_+, \quad X(t) = X(t,\cdot) = X(t,0,\cdot), \quad X(-t) := X(0,t,\cdot).\]
	Let us enumerate a few classical properties of the flow. These are elementary consequences of Duhamel formula and Gronwall inequality.
	\begin{lemma}
		Let $\bu \in L^\infty(\R_+;W^{1,\infty}(\Omega))$ with $\bu|_{\p\Omega} \equiv \bzero.$ The associated characteristic map $X$ satisfies,
		\begin{equation}
		\forall t \in \R, \forall x,y \in \Omega, \quad|X(t,x) - X(t,y)| \leq e^{C|t|\|\nabla \bu\|_{L^\infty(0,t;L^\infty)}}|x-y|. \label{eq:char40}
		\end{equation}
	\end{lemma}
	In particular $X(t)$ is bi-Lipschitz for any $t$. Recall that Liouville theorem ensures that if $\div\,\bu \equiv 0$ the jacobian determinant of $X(t)$ is identically equal to $1$ with respect to $t$. Besides, let us introduce a classical stability estimate on the characteristics.
	\begin{lemma} \label{lem:stabchar}
		Let $\bu_i \in L^\infty(\R_+;W^{1,\infty}(\Omega))$ with ${\bu_i}|_{\p\Omega} \equiv \bzero$ and $\div \, \bu_i \equiv 0$ for $i=1,2$. If $\Omega$ is bounded, the associated characteristic maps $X_i$ satisfy, for any $1\leq q\leq\infty$,
		\begin{equation} \label{eq:stabchar}
		\forall t \in \R_+, \quad \|X_1(t) - X_2(t)\|_{L^q} \leq t e^{Ct\|\nabla \bu_1\|_{L^\infty(0,t; L^\infty)}}\|\bu_1-\bu_2\|_{L^\infty(0,t;L^q)}.
		\end{equation}
		If $q=\infty$, the inequality holds true for unbounded $\Omega$.
	\end{lemma}
	\begin{proof}[Proof] Let us consider $\Omega$ bounded. From Duhamel formula we write for any $t \in \R_+$ and $x\in\Omega$,
		\begin{align*}
			X_1(t,x) - X_2(t,x) &= \int_0^t\bu_1(\tau,X_1(\tau,x))-\bu_2(\tau,X_2(\tau,x)) \ud \tau \\
				& = \int_0^t \bu_1(\tau,X_1(\tau,x)) - \bu_1(\tau,X_2(\tau,x)) \ud \tau \\
					& \qquad + \int_0^t \bu_1(\tau,X_2(\tau,x))-\bu_2(\tau,X_2(\tau,x))\ud\tau.
		\end{align*}
		The Lipschitz regularity of $\bu_i$ and the Minkowski inequality provide, for any $1\leq q\leq \infty$,
		\begin{multline*}
			\|X_1(t)-X_2(t)\|_{L^q} \leq C\|\nabla\bu\|_{L^\infty(0,t;L^\infty)}\int_0^t \|X_1(\tau)-X_2(\tau)\|_{L^q} \ud \tau\\ + \int_0^t\|\bu_1(\tau,X_2(\tau))-\bu_2(\tau,X_2(\tau))\|_{L^q}\ud\tau.
		\end{multline*}
	As a consequence of	Liouville theorem, one has for $1\leq q<\infty$ and $\tau \in [0,t]$,
	\begin{equation*}
		\|\bu_1(\tau,X_2(\tau))-\bu_2(\tau,X_2(\tau))\|_{L^q} = \|\bu_1(\tau,\cdot)-\bu_2(\tau,\cdot))\|_{L^q}.
	\end{equation*}
	The case $q=\infty$ holds naturally true. Then we have
		\begin{multline*}
			\|X_1(t)-X_2(t)\|_{L^q} \leq C\|\nabla\bu_1\|_{L^\infty(0,t;L^\infty)} \int_0^t\|X_1(\tau)-X_2(\tau)\|_{L^q} \ud \tau + t\|\bu_1-\bu_2\|_{L^\infty(0,t;L^q)},
		\end{multline*}
		which yields (\ref{eq:stabchar}) by Gronwall inequality.
	\end{proof}
	
	The classical characteristics method provides the well-posedness of the transport equation.
	\begin{proposition} \label{prop:t}
		Let $\Omega$ be a Lipschitz domain of $\R^d$, not necessarily bounded. Let $\bu \in L^\infty(\R_+;W^{1,\infty}(\Omega))$ with $\bu|_{\p\Omega} \equiv \bzero$, and let $\rho_0 \in L^\infty(\Omega)$. There exists a unique $\rho$ in $L^\infty(\R_+;L^\infty(\Omega))$ satisfying (\ref{eq:t}), which is moreover the push-forward of $\rho_0$ by the characteristic $X$ of $\bu$, namely  
		\begin{equation*}
		\forall t \in \R_+, \quad \rho(t) = \rho_0 \circ X(-t).
		\end{equation*}
	\end{proposition}
	In particular the $L^q$ norm of the solution $\rho$ is constant in time, for any $1\leq q\leq\infty$. Besides, we state the following estimate of the evolution of the difference of two solutions of (\ref{eq:t}) associated to distinct velocity fields. 
	\begin{proposition} \label{prop:stabb}
		Let $\bu_i \in L^\infty(\R_+;W^{1,\infty}(\Omega))$, with ${\bu_i}|_{\p\Omega} \equiv \bzero$ and $\div \,\bu_i \equiv 0$, for $i=1,2$. Let $\rho_0 \in L^\infty(\Omega)$ and $\rho_i$ be the solutions of (\ref{eq:t}) associated to $\bu_i$ with initial datum $\rho_0$. For any $1 < q < \infty$ there exists $\bar{T}(\|\nabla \bu_i\|_{L^\infty})>0$ and $C(\Omega,q)>0$ such that for any $T \in [0,\bar{T}]$,
		\begin{equation*} % \label{eq:stabb}
		 \|\rho_1-\rho_2\|_{L^\infty(0,T;W^{-1,q})} \leq C\|\rho_0\|_{L^\infty}Te^{CT\|\nabla\bu_1\|_{L^\infty(0,T;L^\infty)}}\|\bu_1-\bu_2\|_{L^\infty(0,T;L^q)}.	
		\end{equation*}
	\end{proposition}
	\begin{proof}
		Let $t \in \R_+$ and $\vphi \in C^\infty_\rc(\Omega)$. Since the vector fields $\bu_i$ are divergence-free, the Liouville theorem ensures the following change of variable,
		\begin{align*}
			I_\vphi &:= \int_\Omega \big(\rho_1(t,x)-\rho_2(t,x)\big)\vphi(x) \ud x  \\ &= \int_\Omega \big( \rho_0(X_1(-t,x)) - \rho_0(X_2(-t,x))\big) \vphi(x) \ud x\\
				&= \int_\Omega \rho_0(x) \big( \vphi(X_1(t,x)) - \vphi(X_2(t,x))\big) \ud x.
		\end{align*}
		Since $\vphi$ is smooth, we can write
		\begin{equation*}
			I_\vphi =  \int_\Omega \rho_0(x) \big( X_1(t,x)-X_2(t,x) \big) \cdot \int_0^1 \nabla \vphi (X^\theta(t,x)) \ud \theta \ud x,
		\end{equation*}
		where we set $X^\theta(t,x) := \theta X_1(t,x) + (1-\theta)X_2(t,x)$. Then, Hölder's inequality provides
		\begin{equation} \label{eq:stabbound}
			|I_\vphi| \leq \|\rho_0\|_{L^\infty}\|X_1(t)-X_2(t)\|_{L^q} \int_0^1 \|\nabla\vphi(X^\theta(t))\|_{L^{q'}} \ud\theta.
		\end{equation}
		Let us show that $X^\theta(t)$ is bi-Lipschitz from $\Omega$ onto its range. Consider the derivative of the Duhamel formula satisfied by $X_i$,
		\begin{equation*}
			\nabla X_i(t,\cdot) - I_d = \int_0^t \nabla \bu_i(\tau,X_i(\tau,\cdot))\cdot\nabla X_i(\tau,\cdot) \ud \tau
		\end{equation*}
		and deduce thanks to inequality (\ref{eq:char40}) the uniform estimate
		\begin{equation*}
			\| \nabla X_i(t) - I_d\|_{L^\infty} \leq C\|\nabla\bu_i\|_{L^\infty(0,t;L^\infty)}\left(e^{Ct\|\nabla\bu_i\|_{L^\infty(0,t;L^\infty)}}-1\right).
		\end{equation*}		
		Therefore for some arbitrary constant $C > 1$ there exists $\bar{T}(\|\nabla\bu_i\|_{L^\infty}) > 0$ such that
		\begin{equation*}
			\forall t \in [0,\bar{T}], \quad C^{-1} \leq \det \, \nabla X_i(t) \leq C,
		\end{equation*}
		and such that the Lipschitz constants of $X_i(t)$ are uniformly bounded with respect to $t$, by a constant smaller than $1$, as follows,
		\begin{equation*}
			L_ {\bar{T}} := \max_i \sup_{t\in[0,\bar{T}]} \Lip (X_i(t)-\Id) < 1.
		\end{equation*}
		The latter inequality ensures the injectivity of $X^\theta(t)$, since for any $x,y \in \Omega$, the equality $X^\theta(t,x) = X^\theta(t,y)$ is equivalent to
		\begin{multline*}
			x-y = \theta \Big( X_1(t,y)-y - \big(X_1(t,x)-x\big)\Big) + (1-\theta)\Big(X_2(t,y)-y-\big(X_2(t,x)-x\big)\Big),
		\end{multline*}
		and implies
		\begin{equation*}
			|x-y| \leq L_{\bar{T}} |x-y|,
		\end{equation*}
		so that $x=y$ since $L_{\bar{T}} < 1$. In the end we have proved that $X^\theta(t)$ is bi-Lipschitz for any $\theta\in[0,1]$ and $t \in [0,\bar{T}]$, with a uniform bound on its jacobian determinant, independent of $\theta$ and $t$. Therefore we have for any $1<q<\infty$,
		\begin{equation} \label{eq:hm}
			\|\nabla\vphi(X^\theta(t))\|_{L^{q'}} \leq C_q\|\nabla \vphi\|_{L^q}.
		\end{equation}
		Combining (\ref{eq:stabbound}) and (\ref{eq:hm}) leads to
		\begin{equation*}
			\forall t \in [0,\bar{T}], \quad \int_\Omega(\rho_1(t,x)-\rho_2(t,x))\vphi(x) \ud x \leq C\|\rho_0\|_{L^\infty}\|X_1(t)-X_2(t)\|_{L^q}\|\vphi\|_{W^{1,q'}}.
		\end{equation*}
		Plugging the stability estimate (\ref{eq:stabchar}) and taking the supremum over the test functions provides the following bound on the negative Sobolev norm,
		\begin{equation*}
			\forall t \in [0,\bar{T}], \quad \|\rho_1(t)-\rho_2(t)\|_{W^{-1,q}} \leq C\|\rho_0\|_{L^\infty}t e^{Ct\|\nabla\bu_1\|_{L^\infty(0,t;L^\infty)}}\|\bu_1-\bu_2\|_{L^\infty(0,t;L^q)}.
		\end{equation*}
		It remains to consider the supremum over $t \in [0,T]$ for any $T \leq \bar{T}$ to get the result.
	\end{proof}

\subsection{Proof of Theorem \ref{thm:1}}
	The strategy of the proof is inspired from the one adopted by Mecherbet in \cite{mecherbet} for the space domain $\R^3$. In \cite{mecherbet} the author solves successively Stokes and transport problems, providing a contracting sequence of velocity fields and density profiles. Here the contracting property is obtained by combination of the Stokes estimate from Theorem \ref{thm:galdi} and the stability estimate for the transport from Proposition \ref{prop:stabb}. Both interplay in the case of the whole space thanks to the Green kernel, also called \emph{Oseen tensor}, representation of the Stokes solution and the stability estimates formulated with the Wasserstein distance. Since we work in general open bounded domains, we do not use Green kernels but only rely on the variational estimates. We also replace the Wasserstein distance, handled by Mecherbet, by negative Sobolev norms. Both are known to be related, see \cite[subsection 5.5.2]{santambrogio}, and in our case present the same asset to allow the statement of stability estimates without any derivability assumption concerning the density.
	
	\paragraph{Local existence :} Set $\rho^0 \equiv \rho_0$ in $L^\infty(\R_+; L^\infty(\Omega))$. Theorem \ref{thm:galdi} and Proposition \ref{prop:t} allow to define the following  sequences by induction on $N \in \N$,
	\[ \rho^N \in L^\infty(\R_+;L^\infty(\Omega)), \quad \bu^N \in L^\infty(\R_+;W^{1,\infty}(\Omega)), \]
	satisfying for any $N \in \N$ the Stokes problem
	\begin{equation} \label{eq:TN}
	\left\{
	\begin{array}{rcll}
		-\Delta\bu^N + \nabla p^N &=& -\rho^N\be_z &\text{in } \R_+\times \Omega, \\
		\div\,\bu^N &=& 0 &\text{in } \R_+\times\Omega,\\
		\bu^N &=& \bzero &\text{in } \R_+\times\p\Omega,
	\end{array}
	\right.
	\end{equation}
	and the transport equation
	\begin{equation} \label{eq:SN}
		\left\{
		\begin{array}{rcll}
		\p_t \rho^{N+1} + \bu^N\cdot\nabla\rho^{N+1} &=& 0 &\text{in } \R_+\times\Omega, \\
		\rho^{N+1}|_{t=0} &=& \rho_0 &\text{in } \Omega.
		\end{array}
		\right.
	\end{equation}
	Let us denote $B := C \|\rho_0\|_{L^\infty}$ with adjustable constant. Since $\rho^N$ is the push-forward of $\rho_0$ by the flow of $\bu^{N-1}$, we have the uniform bound
	\begin{equation*}
		\forall N, \quad \|\rho^N\|_{L^\infty(\R_+;L^\infty)} = \|\rho_0\|_{L^\infty} \leq B.
	\end{equation*}
	Using althemore Stokes estimate (\ref{eq:sestub}), we obtain
	\begin{equation*}
		\forall N, \quad \|\bu^N\|_{L^\infty(\R_+;W^{1,\infty})} \leq C\|\rho^N\|_{L^\infty(\R_+,L^\infty)} \leq B.
	\end{equation*}
	Hence $\rho^N, \bu^N$ and $\nabla\bu^N$ converge in $\rw^*-L^\infty(\R_+\times\Omega)$ up to the extraction of subsequences. Besides, estimates from Proposition \ref{prop:stabb} and Theorem \ref{thm:galdi} ensure that there exists $\bar{T}(\|\rho_0\|_{L^\infty}) >0$ such that
	\begin{align}
		\forall T \in [0,\bar{T}], \quad \|\rho^{N+1}-\rho^N\|_{L^\infty(0,T;H^{-1})} 
			& \leq BTe^{BT} \|\bu^N-\bu^{N-1}\|_{L^\infty(0,T;H^1)} \notag \\
			& \leq BTe^{BT} \|\rho^N-\rho^{N-1}\|_{L^\infty(0,T;H^{-1})}. \label{eq:loop}
	\end{align}
	We see that for $T > 0$ small enough we have $BTe^{BT} < 1$ so that $(\rho^N)_N$ is a Cauchy sequence in $L^\infty(0,T;H^{-1}(\Omega))$. As a consequence of the Stokes estimate (\ref{eq:sest}), we have that $(\bu^N)_N$ is also a Cauchy sequence, in $L^\infty(0,T;H^1(\Omega))$. Its limit belongs to $L^\infty(0,T;W^{1,\infty}(\Omega))$ since it converges for the $\weak^*$ topology. In particular, $\bu^N$ converges in $L^1(0,T;W^{1,1}(\Omega))$, which together with the $\weak^*$ convergence of $(\rho^N)_N$ allows to pass to the limit in the weak formulations of both (\ref{eq:TN}) and (\ref{eq:SN}). Therefore the limit $(\rho,\bu)$ satisfies (\ref{eq:st}) on a short time, with regularity
	\begin{equation*}
		L^\infty(0,T;L^\infty(\Omega)) \times L^\infty(0,T;W^{1,\infty}(\Omega)).
	\end{equation*}
	
	\paragraph{Local uniqueness :} Let $(\rho_i,\bu_i)$ be two such solutions of (\ref{eq:st}). The estimate (\ref{eq:loop}) adapts in
	\begin{equation*}
		\|\rho_1-\rho_2\|_{L^\infty(0,T;H^{-1})} \leq BTe^{BT}\|\rho_1-\rho_2\|_{L^\infty(0,T;H^{-1})},
	\end{equation*}
	up to the choice of a smaller $T>0$. We deduce that $\rho_1 = \rho_2$ on $[0,T]$, and that $\bu_1 = \bu_2$ thanks to the Stokes estimate. 
	
	\paragraph{Globality :} By existence and uniqueness of a solution to (\ref{eq:st}) locally in time, we know that there exists a unique maximal solution $(\rho,\bu)$ on some interval $[0,T^*)$ with $T^* \in [0,\infty]$. Remark that $T^*$ depends only on $\|\rho_0\|_{L^\infty}$. Since $\|\rho\|_{L^\infty} \equiv \|\rho_0\|_{L^\infty}$, a classical continuation argument implies that $T^*= \infty$, which proves that the solution is global.
	
	\subsection{Stability estimate for the system in a bounded domain}
		We prove a stability estimate for the Stokes-transport system, inherited from Proposition \ref{prop:stabb}.
	\begin{proposition} \label{prop:stabsys}
		Let $\rho_{0,i} \in L^\infty(\Omega)$ and set $\rho_i$ the solution of (\ref{eq:st}) with initial datum $\rho_{0,i}$ for $i = 1,2$. For any $1 < q < \infty$ there exists $C(\Omega,q, \|\rho_{0,i}\|_{L^\infty}) > 0$ such that
		\begin{equation} \label{eq:stabst}
			\forall T \in \R_+, \quad \|\rho_1-\rho_2\|_{L^\infty(0,T;W^{-1,q})} \leq Ce^{CT} \|\rho_{0,1}-\rho_{0,2}\|_{W^{-1,q}}.
		\end{equation}
	\end{proposition}
	
	\begin{proof}
		Let us set $\bu_i$ the velocity fields associated to $\rho_i$ for $i=1,2$. Set $\rho_{1,2}$ the solution of (\ref{eq:t}) with initial datum $\rho_{0,1}$ and vector field $\bu_2$. Hence, consider the triangular inequality
		\begin{equation*}
			\forall T \in \R_+, \quad \|\rho_1-\rho_2\|_{L^\infty(0,T;W^{-1,q})} \leq \underbrace{\|\rho_1-\rho_{1,2}\|_{L^\infty(0,T;W^{-1,q})}}_{I_1} + \underbrace{\|\rho_{1,2}-\rho_2\|_{L^\infty(0,T;W^{-1,q})}}_{I_2}
		\end{equation*}
		From Proposition \ref{prop:stabb} and estimate (\ref{eq:sestub}) we know that there exists some $\bar{T}(\|\rho_{0,1}\|_{L^\infty}) > 0$ and $C(\Omega,q)>0$ such that
		\begin{equation*}
			\forall T \in [0,\bar{T}], \quad I_1 \leq C\|\rho_{0,1}\|_{L^\infty}e^{CT\|\nabla\bu_1\|_{L^\infty(0,T;L^\infty)}}\|\bu_1-\bu_2\|_{L^\infty(0,T;L^q)}.
		\end{equation*}
		Let us denote $B := C\max_i \|\rho_{0,i}\|_{L^\infty}$. Stokes estimates (\ref{eq:sestd}) and (\ref{eq:sestub}) respectively provide here 
		\begin{equation*}
			\|\bu_1-\bu_2\|_{L^\infty(0,T;L^q)} \leq C \|\rho_1-\rho_2\|_{L^\infty(0,T;W^{-1,q})}, \qquad \|\nabla \bu_1\|_{L^\infty(0,T;L^\infty)} \leq B,
		\end{equation*}
		which yields
		\begin{equation*}
			\forall T \in [0,\bar{T}], \quad I_1 \leq BTe^{BT}\|\rho_1-\rho_2\|_{L^\infty(0,T;W^{-1,q})}.
		\end{equation*}
		To bound $I_2$ let us apply again Liouville theorem for any $t \in \R_+$ and $\vphi \in C^\infty_\rc(\Omega)$, to get
		\begin{equation*}
			\int_\Omega\big(\rho_{1,2}(t,x) -\rho_2(t,x)\big) \vphi(x)\ud x = \int_\Omega (\rho_{0,1}(x)-\rho_{0,2}(x))\vphi(X_2(t,x))\ud x.
		\end{equation*}
		Now we have, by definition of the Sobolev norm,
		\begin{equation*}
			\int_\Omega\big(\rho_{1,2}(t,x) -\rho_2(t,x)\big) \vphi(x)\ud x \leq \|\rho_{0,1}-\rho_{0,2}\|_{W^{-1,q}} \| \vphi(X_2(t))\|_{W^{1,q'}}.
		\end{equation*}
		From estimates (\ref{eq:char40}) and (\ref{eq:sestub}) together with the bound $\|\nabla X_i\|_{L^\infty(0,t;L^\infty)} \leq Ce^{Bt}$ provided by Lemma \ref{lem:stabchar}, it follows
		\begin{equation*}
			\|\vphi(X_2(t))\|_{W^{1,q'}} \leq Ce^{Bt} \|\vphi\|_{W^{1,q'}}.
		\end{equation*}
		Passing to the supremum over the test functions gives
		\begin{equation*}
			I_2 \leq Ce^{BT} \|\rho_{0,1}-\rho_{0,2}\|_{W^{-1,q}}.
		\end{equation*}
		In the end we have shown that
		\begin{multline*}
			\forall T \in [0,\bar{T}],\quad	\|\rho_1-\rho_2\|_{L^\infty(0,T;W^{-1,q})} \leq BTe^{BT} \|\rho_1-\rho_2\|_{L^\infty(0,T;W^{-1,q})} \\ + Ce^{BT} \|\rho_{0,1}-\rho_{0,2}\|_{L^\infty(0,T;W^{-1,q})}.
		\end{multline*}
		Up to the choice of a small enough $\bar{T} > 0$, we have
		\begin{equation*}
			\forall T\in[0,\bar{T}],\quad \|\rho_1-\rho_2\|_{L^\infty(0,T;W^{-1,q})} \leq \underbrace{\tfrac{ Ce^{B\bar{T}}}{1- B\bar{T}e^{B\bar{T}}}}_{\bar{C}} \|\rho_{0,1}-\rho_{0,2}\|_{W^{-1,q}},
		\end{equation*}
		Notice that the choice of $\bar{T}$ depends only on $B$, and recall that $\|\rho_i(t)\|_{L^\infty} = \|\rho_{0,i}\|_{L^\infty}$ for any $t \in \R_+$. Therefore, one obtains by induction
		\begin{equation*}
			\forall T \in \R_+, \quad \|\rho_1-\rho_2\|_{L^\infty(0,T;W^{-1,q})} \leq \bar{C}^{\lceil T/\bar{T} \rceil} \|\rho_{0,1}-\rho_{0,2}\|_{W^{-1,q}}, 
		\end{equation*}
		which also writes (\ref{eq:stabst}).
	\end{proof}

\section{Well-posedness of the system in the infinite strip} \label{sec:unbounded}
	
		In this section, $\Omega$ stands for the \emph{infinite strip} $\R\times(0,1)$. We denote by $(\be_x,\be_z)$ the canonical base of $\R^2$ in which $\bu$ has coordinates $(u_1,u_2)$. 
		
		Regarding our problem, the transport theory does not depend on the nature of the domain, and the related results presented in the previous section are still valid in the strip. The main difference lies in the tools and methods required to solve the Stokes equation in $\Omega$. In particular, we state that this equation is still well-posed for $L^\infty$ data, with $W^{1,\infty}$ solution. To do so, we consider \emph{Kato spaces}, also known as \emph{uniformly local Sobolev spaces}. This framework allows to consider uniformly bounded densities $\rho$ and non globally integrable velocity fields $\bu$, having infinite energy $\|\bu\|_{H^1}$, but admitting locally a finite energy, uniformly bounded with respect to any compact stallion subdomain  of the strip. We first solve Stokes equation in a $L^2$ framework, then recover elliptic regularity and use Sobolev injection to prove its well-posedness for bounded data. Even if it requires known methods, we did not find the precise proof of this latter result in the uniformly local framework.
		
		The Subsection \ref{sbsec:Stokesstrip} is dedicated to the statements and proofs related to the Stokes equation. In particular we introduce Kato spaces in paragraph \ref{sbsbsec:spaces}, then discuss the flux condition in paragraph \ref{sbsbsec:flux} and finally prove the related well-posedness theorems in paragraph \ref{sbsbsec:wp}. The Subsection \ref{sbsec:stabstrip} concerns the stability estimate for the transport in the strip. Subsection \ref{sbsec:proofstrip} contains the proof of the well-posedness of the coupling, and we state its stability estimate in Subsection \ref{sbsec:stabstripsys}.
	
	\subsection{The Stokes problem in the strip} \label{sbsec:Stokesstrip}
	
		\subsubsection{The functional spaces} \label{sbsbsec:spaces}
		Let us set the following subdomains of $\Omega$, for any $k \in \Z$,
		\begin{equation*}
			U_k = \{ (x,z) \in \Omega : k < x < k+1 \} \qquad U_k^* = \{ (x,z) \in \Omega : k-1 < x < k+2 \}.
		\end{equation*}
		Define a smooth map $\chi : \Omega \rightarrow [0,1]$, depending on $x$ only, equal to $1$ on $U_0$ and to $0$ outside $U_0^*$. Set its translations $\chi_k := \chi(\cdot - k\be_x)$ so that $\chi_k$ is equal to $1$ on $U_k$ and supported in $U_k^*$. For convenience we choose $\chi$ so that $\sum_{k= -\infty}^{+\infty} \chi_k = 2$ in $\Omega$.
		Let us set for $m \in \Z$ and $1 \leq q \leq \infty$ the uniformly local norm
		\begin{equation*}
		\forall u \in W^{m,q}_\loc(\Omega), \qquad \|u\|_{W^{m,q}_\uloc} := \sup_{k\in\Z} \|\chi_k u\|_{W^{m,q}},
		\end{equation*}
		and define the \emph{Kato space} as the set of locally Sobolev maps having a finite uniformly local norm,
		\begin{equation*}
			W^{m,q}_\uloc(\Omega) := \{ u \in W^{m,q}_\loc(\Omega) : \|u\|_{W^{m,q}_\uloc} < \infty \}.
		\end{equation*}
		This is a Banach space, that does not depend on the choice of $\chi$, see for instance \cite[\S 2.2]{Alazard_2016}. The following result provides in particular some handy norms equivalences.
		\begin{lemma} \label{lem:comp}
			For any $m \in \N$ and $1<q\leq\infty$, the following quantities are equivalent,
			\begin{equation*}
			u \in W^{m,q}_\loc(\Omega), \quad \sup_{k\in\Z} \|u\|_{W^{m,q}(U_k)} \simeq \sup_{k\in\Z} \|u\|_{W^{m,q}(U_k^*)} \simeq \|u\|_{W^{m,q}_\uloc}.
			\end{equation*}
			For any $m \in \N^*$ and $1<q<\infty$, the following quantities are comparable,
			\begin{equation*}
			u \in W^{-m,q}_\loc(\Omega), \quad \sup_{k\in\Z} \|u\|_{W^{-m,q}(U_k)} \lesssim \sup_{k\in\Z} \|u\|_{W^{-m,q}(U_k^*)} \simeq \|u\|_{W^{-m,q}_\uloc}.
			\end{equation*}
		\end{lemma}
		Although this result presents no difficulty, we provide a short proof and a comment, about the missing inequality for the last comparison, in appendix \ref{ap:comp} for sake of completeness. 
		
		\subsubsection{Flux condition} \label{sbsbsec:flux}
		
		In general, the homogeneous Stokes system as formulated in (\ref{eq:s}) admits non-trivial solutions in domains with unbounded boundaries, called \emph{Poiseuille solutions}, see for instance \cite[Section IV]{galdi}. In our case, these are described as follows
		\begin{equation*}
		\bu_\phi(x,z) = \begin{pmatrix} 6\phi z(1-z) \\ 0 \end{pmatrix}, \quad p_\phi(x,z) = 12\phi x, \qquad \phi \in \R.
		\end{equation*}
		Let us introduce the definition of the \emph{flux} of $\bu$ throw the section of abscissa $x \in \R$ of $\Omega$, 
		\begin{equation*}
		\int u_1 \ud z := \int_0^1 u_1(x,z) \ud z = \int_0^1 \bu(x,z) \cdot \be_x \ud z.
		\end{equation*}
		Notice that the divergence-free and the homogeneous Dirichlet condition ensure that any solution of (\ref{eq:s}) on $\Omega$ has flux independent of $x$, which we will denote further by $\int u_1 \ud z$. Indeed,
		\begin{equation} \label{eq:zf}
		\frac{\rd}{\rd x} \left[\int_0^1 u_1(x,z) \ud z\right] = \int_0^1 \p_x u_1(x,z)\ud z = - \int_0^1 \p_z u_2(x,z) \ud z = u_2(x,0) - u_2(x,1) = 0.
		\end{equation}
		For instance, the flux of the Poiseuille solution $\bu_\phi$ is $\phi$. We will see that a choice of flux value prescribes a unique Poiseuille solution, and provides uniqueness of a solution in Kato spaces. Since the Stokes equation is linear, we can choose one value without loss of generality. From now on we consider the Stokes problem with the zero flux condition,
		\begin{equation} \label{eq:S0}
		\left\{
		\begin{array}{rcll}
		-\Delta \bu + \nabla p &= &\bf &\text{in } \Omega, \\
		\div \, \bu &= &0 &\text{in } \Omega,\\
		\bu &= &\bzero & \text{on } \p\Omega, \\
		\int u_1 \ud z &= &0 & \text{in } \R.
		\end{array}
		\right.
		\end{equation}
		
		\subsubsection{Well-posedness of the Stokes problem in the strip} \label{sbsbsec:wp}
		
		\indent We show that the system (\ref{eq:S0}) is well-posed for $H^{-1}_\uloc(\Omega)$ and $L^2_\uloc(\Omega)$ data $\bf$, with some elliptic regularity gain. We then deduce that this system is also well-posed for $L^\infty(\Omega)$ data with $W^{1,\infty}(\Omega)$ solutions, by adapting the steps of the proof of the bounded domain case. The general technique presented here is originally due to Ladyzhenskaya and Solonnikov \cite{Ladyzhenskaya_1983}. A proof in a framework closer to ours can be found in \cite[Theorem 3]{gerard2010relevance}. Although these are classical tools, we have not found the proof of this result in the literature.
		\begin{theorem} \label{thm:suloc}
			Let $\bf \in H^{-1}_\uloc(\Omega)$. There exists a unique $\bu$ in $H^1_\uloc(\Omega)$ satisfying (\ref{eq:S0}), which moreover obeys the inequality
			\begin{equation} \label{eq:suloc}
				\|\bu\|_{H^1_\uloc} \leq C\|\bf\|_{H^{-1}_\uloc}.
			\end{equation}
		\end{theorem}
		Let us introduce a few more notations. Set for any $k \in \Z$,
		\begin{equation*}
		\Omega_k := \{(x,z) \in \Omega : -k < x < k\}.
		\end{equation*}
		Let us define for any $k \in \N^*$ a smooth map $\eta_k : \Omega \rightarrow [0,1]$, depending on $x$ only, equal to $1$ in $\Omega_k$ and supported in $\Omega_{k+1}$. Remark that its derivatives are supported in $\Omega_{k+1}\backslash \Omega_k$. We can choose $\eta_k$ such that there exists a constant $C>0$ independent of $k$ and satisfying
		\begin{equation*}
			\|\eta_k'\|_{L^\infty} + \|\eta_k''\|_{L^\infty} \leq C.
		\end{equation*}
		Finally, let us observe the following estimate linking uniformly local and classical Sobolev norms over $\Omega_n$. We report its proof in appendix \ref{ap:gluloc}.
		\begin{lemma} \label{lem:gluloc}
			Let $n \in \N^*$. There exists a constant $C>0$ such that for any $f$ in $L^2_\uloc(\Omega)$, resp. in $H^{-1}_\uloc(\Omega)$, one has
			\begin{equation*}
			\|f\|_{L^2(\Omega_n)} \leq C n^{1/2} \|f\|_{L^2_\uloc}, \quad \text{resp. } \|f\|_{H^{-1}(\Omega_n)} \leq Cn^{1/2}\|f\|_{H^{-1}_\uloc}.
			\end{equation*}
		\end{lemma}
			
		\begin{proof}[Proof of Theorem \ref{thm:suloc}]
			Let us set for any $n\in\N^*$ the unique couple $(\bu_n,p_n)$ in $H^1(\Omega_n)\times \left( L^2(\Omega_n) \slash \R \right)$ satisfying the system
			\begin{equation*} \label{eq:Sn} %\tag{$\text{S}_n$}
				\left\{
				\begin{array}{rcll}
					-\Delta\bu_n + \nabla p_n &= &\bf &\text{in } \Omega_n,\\
					\div \, \bu_n &= &0 &\text{in } \Omega_n, \\
					\bu_n &= &\bzero &\text{on } \p\Omega_n,
				\end{array}
				\right.
			\end{equation*}
			existence and uniqueness of which is ensured by \cite[Theorem IV.5.1]{boyer2012mathematical}. Then define for any integers $1 \leq k \leq n$ the energy of $\bu_n$ on the subdomain $\Omega_k$,
			\begin{equation*}
				E_{n,k} := \|\bu_n\|^2_{H^1(\Omega_k)} = \int_{-k}^k\int_0^1 |\nabla\bu_n|^2 + |\bu_n|^2.
			\end{equation*}
			By evaluating (\ref{eq:Sn}) in the test function $\bu_n$ we find
			\begin{equation*}
				\int_{-n}^n\int_0^1 |\nabla \bu_n|^2 = \langle \bf,\bu_n \rangle_{\Omega_n} \leq \|\bf\|_{H^{-1}(\Omega_n)} \|\bu_n\|_{H^1(\Omega_n)}.
			\end{equation*}
			Using Lemma \ref{lem:gluloc} and Poincaré's inequality, for which the constant involved can be chosen independent of $n$, one finds
			\begin{equation*}
				E_{n,n} \leq C n^{1/2} \|\bf\|_{H^{-1}_\uloc}.
			\end{equation*}
			Our goal is to show that there exists $C>0$ independent of $n$ and $\bf$, such that
			\begin{equation} \label{eq:En1}
				E_{n,1} = \|\bu_n\|_{H^1(\Omega_1)}^2 \leq C\|\bf\|_{H^{-1}_\uloc}^2.
			\end{equation}
			This allows to conclude to the existence of a solution $\bu \in H^1_\uloc(\Omega)$ obeying the estimate (\ref{eq:suloc}), by translation invariance of the domain and compactness considerations. To prove (\ref{eq:En1}) we will fix $n\in\N^*$ and show by descending induction over $k$ that there exists $C>0$ independent of $\bu$, $f$, $n$ and $k$ such that
			\begin{equation*}
				\forall \,1 \leq k \leq n, \quad E_{n,k} \leq Ck\|\bf\|_{H^{-1}_\uloc}^2.
			\end{equation*}
			
			Let us evaluate the variational formulation of (\ref{eq:Sn}) in the test function $\eta_k\bu_n$, which yields
			\begin{equation} \label{eq:dev}
				\int_{\Omega_n} \eta_k|\nabla\bu_n|^2 = \langle\bf,\eta_k\bu_n\rangle_{\Omega_{k+1}} - \int_{\Omega_n} \eta_k'\bu_n\cdot\p_x\bu_n +  \int_{\Omega_n} \eta_k'p_n u_{n,1}.
			\end{equation}
			By Poincaré's inequality we bound from below the left hand side by $E_{n,k}$, up to a multiplicative constant. Let us bound from above all the right hand side terms. Lemma \ref{lem:gluloc} provides
			\begin{equation} \label{eq:bound1}
				\langle \bf,\eta_k\bu_n\rangle_{\Omega_{k+1}} \leq C\|\bf\|_{H^{-1}(\Omega_{k+1})}\|\bu_n\|_{H^1(\Omega_{k+1})} \leq C(k+1)^{1/2} \|\bf\|_{H^{-1}_\uloc} E_{n,k+1}^{1/2}.
			\end{equation}
			Since $\eta_k'$ is supported in $\Omega_{k+1}\backslash \Omega_k$ and uniformly bounded independently of $k$, we have
			\begin{equation} \label{eq:bound2}
				\left|\int_{\Omega_n}\eta_k' \bu_n\cdot\p_x\bu_n \right| \leq C\int_{\Omega_{k+1}\backslash \Omega_k}|\nabla\bu_n|^2 +|\bu_n|^2 = C(E_{n,k+1}-E_{n,k}).
			\end{equation}
			Let us split the remaining integral of \ref{eq:dev} as follows,
			\begin{equation*}
				\int_{\Omega_n}\eta_k'p_nu_{n,1} = \int_{U_k}\eta_k'p_nu_{n,1} + \int_{-U_k}\eta_k'p_nu_{n,1}.
			\end{equation*}
			Remark that
			\begin{equation*}
				\int_{U_k}\eta_k' u_{n,1} = \int_k^{k+1} \eta_k' \int u_{n,1} \ud z = 0,
			\end{equation*}
			since the flux $\int \bu_{n,1} \ud z$ is independent of $x$ for the same reason as in (\ref{eq:zf}), and equal to $0$ because of the homogenous Dirichlet condition in $\{ x = \pm n \}$. Hence, let us denote by $\langle p_n\rangle_{U_k}$ the average of $p_n$ over $U_k$, and find
			\begin{align*}
				\int_{U_k} \eta_k'p_nu_{n,1}
					= \int_{U_k} \eta_k'(p_n-\langle p_n\rangle_{U_k})u_{n,1}
					\leq C\|p_n-\langle p_n\rangle_{U_k}\|_{L^2(U_k)}\|\bu_n\|_{L^2(U_k)}.
			\end{align*}
			Let us apply Ne\v{c}as inequality, see \cite[Lemma IV.1.9]{boyer2012mathematical}, and get
			\begin{align*}
				\|p_n-\langle p_n\rangle_{U_k}\|_{L^2(U_k)} &\leq C\|\nabla p_n\|_{H^{-1}(U_k)} \\
				&\leq C\|\Delta \bu_n + \bf\|_{H^{-1}(U_k)} \\
				&\leq C\left( \|\nabla \bu_n\|_{L^2(U_k)} + \|\bf\|_{H^{-1}(U_k)}\right).
			\end{align*}
			Hence we deduce, using Lemma \ref{lem:gluloc}, that
			\begin{equation} \label{eq:bound3}
				\left|\int_{U_k} \eta_k'p_nu_{n,1}\right| \leq C\big( (E_{n,k+1}- E_{n,k})^{1/2} + \|\bf\|_{H^{-1}_\uloc}\big)E_{n,k}^{1/2}.
			\end{equation}
			The very same considerations hold true for the integral over $-U_k$. Bounding (\ref{eq:dev}) thanks to (\ref{eq:bound1}), (\ref{eq:bound2}) and (\ref{eq:bound3}), plus applying Young's inequality, we obtain that for any integers $k,n$ such that $1 \leq k \leq n$,
			\begin{equation} \label{eq:rel}
				E_{n,k} \leq C\left( E_{n,k+1} - E_{n,k} + (k+1)\|\bf\|_{H^{-1}_\uloc}^2\right).
			\end{equation}
			This relation implies (\ref{eq:En1}), as stated in the following lemma, proven apart in appendix \ref{ap:rel}.
			\begin{lemma}\label{lem:tec}
				Let $(E_{n,k})_{k,n}$ be a non-negative family indexed by all couples $(k,n) \in \N^2$ satisfying $1 \leq k \leq n$, non-decreasing with respect to $k$, obeying (\ref{eq:rel}) and such that
				\begin{equation*}
					\forall n \in \N^*, \quad E_{n,n} \leq Cn\|\bf\|_{H^{-1}_\uloc}^2.
				\end{equation*}
				There exists $C_0 > 0$ and $k_0 \in \N^*$ independent of $\bf$ such that for any $k,n \in \N$ satisfying $k_0 \leq k \leq n$, we have
				\begin{equation*}
					E_{n,k} \leq C_0 k \|\bf\|_{H^{-1}_\uloc}^2.
				\end{equation*}
			\end{lemma}
			\noindent This result implies the expected inequality,
			\begin{equation*}
				\forall n \geq k_0, \quad E_{n,1} \leq E_{n,k_0} \leq C_0k_0 \|\bf\|_{H^{-1}_\uloc}^2.
			\end{equation*}
			By extending $\bf|_{\Omega_n}$ and $\bu_n$ to $\Omega$ by $\bzero$ outside $\Omega_n$, we can perform a similar analysis and find the same energy estimates over each subdomain $U_\ell$, namely,
			\begin{equation*}
				\forall n \geq k_0, \forall \ell \in \Z,  \quad \|\bu_n\|_{H^1(U_\ell)}^2 \leq C \|\bf\|_{H^{-1}_\uloc}^2,
			\end{equation*}
			where $C = C_0k_0$ with $C_0$ and $k_0$ independent of $n$ and $\ell$. Therefore, for any $n\geq k_0$, $\bu_n$ belongs to $H^1_\uloc(\Omega)$ and satisfies
			\begin{equation*}
				\|\bu_n\|_{H^1_\uloc}\leq C\|\bf\|_{H^{-1}_\uloc}.
			\end{equation*}
			Since bounded subsets of $H^1(U_\ell)$ are weakly relatively compact, there exists a subsequence of $(\bu_n)_n$ converging weakly in $H^1_\loc(\Omega)$ toward some $\bu \in H^1_\uloc(\Omega)$, with $\bu$ satisfying estimate (\ref{eq:suloc}). The limit also verifies $\int u_1 \ud z=0$ since every $\bu_n$ has zero flux. Hence, it is a solution of (\ref{eq:S0}).
			
			To prove uniqueness of such an element, let us consider some $\bu \in H^1_\uloc(\Omega)$ satisfying (\ref{eq:S0}) with $\bf = \bzero$. Define the energy $E_k := \|\bu\|_{H^1(\Omega_k)}^2$ and proceed to the same computations as previously to find 
			\begin{equation*}
				E_k \leq C(E_{k+1}-E_k + 1).
			\end{equation*}
			Notice that the zero flux condition is necessary to bound the pressure term and obtain such an estimate. Since $E_{k+1}-E_k$ is bounded by $\|\bu\|_{H^1_\uloc}^2$, we have
			\begin{equation*}
				\forall k \in \Z, \quad E_k \leq C( \|\bu\|_{H^1_\uloc}^2 +1) < \infty,
			\end{equation*}
			which means that $\bu$ belongs to $H^1(\Omega)$. Then we conclude by well-posedness of the Stokes system over $\Omega$ in $H^1(\Omega)$, see for instance \cite{temam2001navier}. 
		\end{proof}
	
		\begin{remark}
			The pressure does not belong to $L^2_\uloc(\Omega)$ in general; observe for instance the following triplet, satisfying (\ref{eq:S0}),
			\begin{equation*}
				\bf = \be_x, \quad \bu = \bzero, \quad p(x,z) = x.
			\end{equation*}
			Nevertheless, we have thanks to Ne\v{c}as inequality some similar estimate as (\ref{eq:suloc}) on the pressure,
			\begin{equation*}
				\sup_{k\in\Z} \|p-\langle p \rangle_{U_k}\|_{L^2(U_k)} \leq C\|\bf\|_{L^2_\uloc}.
			\end{equation*}	
		\end{remark}
		\begin{remark} \label{rk:lay}
			This proof does not adapt straightforwardly to the case of the layer domain $\R^2\times(0,1)$. The first issue one needs to deal with is to determine conditions on $\bu$ ensuring uniqueness of a solution. Also, the descending induction on the energy estimates no longer holds in this form. Indeed, one needs to replace the slices $[k,k+1]\times(0,1)$ in $\R\times(0,1)$ by chunks $[k,k+1]\times[\ell,\ell+1]\times(0,1)$ in $\R^2\times(0,1)$. To bound the energy on $[-k,k]^2\times(0,1)$ by the energy on $[-(k+1),k+1]^2\times(0,1)$ makes appear some quadratic terms in $k$ in (\ref{eq:rel} instead of the linear ones present for the strip case, making the induction fail. Under different boundary assumptions, it is however possible to adapt it in a non-trivial way and to conclude; see for instance \cite[Section 3]{dalibard}.
		\end{remark}
		Since $L^\infty(\Omega) \subset L^2_\uloc(\Omega) \subset H^{-1}_\uloc(\Omega)$, we already have existence of a solution to (\ref{eq:S0}) for $L^\infty$ data. Recall that we need to establish the $W^{1,\infty}$ regularity of this solution. We first show that the system satisfies some elliptic regularity property in the hilbertian framework.
		\begin{theorem} \label{thm:ellreg}
			Let $\bf\in L^2_\uloc(\Omega)$. The associated solution $\bu \in H^1_\uloc(\Omega)$ of (\ref{eq:S0}) belongs to $H^2_\uloc(\Omega)$ and obeys the inequality
			\begin{equation} \label{eq:sulocell}
				\|\bu\|_{H^2_\uloc}\leq C\|\bf\|_{L^2_\uloc}.
			\end{equation}
		\end{theorem}
		\begin{proof}
			The demonstration consists in truncating the global solution $\bu$ within some bounded subdomains, and to use the elliptic regularity in these bounded domains provided by Theorem \ref{thm:galdi}. Let $\bf \in L^2_\uloc(\Omega)$ and set $(\bu,p)\in H^1_\uloc(\Omega)\times \left(L^2_\loc(\Omega)/\R\right)$, the associated solution to (\ref{eq:S0}). For any $k \in \Z$, set $\bu_k := \chi_k \bu$ and $q_k := \chi_k (p-\langle p \rangle_{U_k^*})$, which satisfy the system
			\begin{equation} \label{eq:trunc}
				\left\{
				\begin{array}{rcll}
					-\Delta \bu_k + \nabla q_k &= &\bF_k &\text{in } \tilde{U}_k,\\
					\div\,\bu_k &= &\chi_k'u_1 &\text{in } \tilde{U}_k, \\
					{\bu_k} &= &\bzero & \text{on } \p \tilde{U}_k,
				\end{array}
				\right.
			\end{equation}
			where we set
			\begin{equation*}
				\bF_k := \chi_k \bf - 2\chi_k'\p_x\bu -\chi_k'' \bu + \chi_k'\left(p-\langle p \rangle_{U_k^*}\right) \be_x 
			\end{equation*}
			for any smooth bounded subdomain $\tilde{U}_k$ of $\Omega$ containing $U_k^*$. Let us set $( \tilde{U}_k )_k$ a family of such domains given by the choice of a smooth $\tilde{U}_0$ containing $U_0^*$ and its translations $\tilde{U}_k = \tilde{U}_0 + k\be_x$. The regularity of $\bu$ and $p$ implies that $\bF_k$ belongs to $L^2(\tilde{U}_k)$ and that $\chi_k'u_1$ satisfies the compatibility condition (\ref{eq:compcond}).  Therefore, Theorem \ref{thm:galdi} ensures that $\bu_k$ is the only solution of (\ref{eq:trunc}) on $\tilde{U}_k$, with estimate
			\begin{equation} \label{eq:locest}
				\|\bu_k\|_{H^2(\tilde{U}_k)} \leq C\|\bF_k\|_{L^2(\tilde{U}_k)},
			\end{equation}
			where the constant $C>0$ can be chosen independent of $k$ since the subdomains $\tilde{U}_k$ are translations of each other. A few computations lead to
			\begin{equation} \label{eq:Fk}
				\|\bF_k\|_{L^2(\tilde{U}_k)} \leq C ( \|\bu\|_{H^1_\uloc} + \|\bf\|_{L^2_\uloc} + \|p-\langle p\rangle_{U_k^*} \|_{L^2(U_k^*)}).
			\end{equation}
			Ne\v{c}as inequality and Lemma \ref{lem:gluloc} provide
			\begin{align*}
				\|p-\langle p\rangle_{U_k^*} \|_{L^2(U_k^*)} &\leq C\|\nabla p\|_{H^{-1}(U_k^*)} \\
					&\leq C\|\Delta \bu + \bf\|_{H^{-1}(U_k^*)} \\
					&\leq C(\|\bu\|_{H^1_\uloc} + \|\bf\|_{L^2_\uloc}).
			\end{align*}
			The latter estimate combined with (\ref{eq:suloc}) and (\ref{eq:Fk}) in (\ref{eq:locest}) leads to
			\begin{equation*}
				\|\bu\|_{H^2(U_k)} \leq \|\bu_k\|_{H^2(\tilde{U}_k)} \leq C\|\bf\|_{L^2_\uloc},
			\end{equation*}
			which proves that $\bu$ belongs to $H^2_\uloc(\Omega)$ and satisfies inequality (\ref{eq:sulocell}).
		\end{proof}
		From the latter result and Sobolev embeddings we obtain existence of $W^{1,q}_\uloc$ solutions for $L^q$ data. Then elliptic regularity is once again recovered and we show that these solutions are in $W^{2,q}_\uloc$. Finally Sobolev embeddings once again yield a unique solution in $W^{1,\infty}$ for $L^\infty$ data.
		\begin{theorem} \label{thm:gwplip}
			Let $\bf \in L^\infty(\Omega)$. There exists a unique $\bu \in W^{1,\infty}(\Omega)$ satisfying (\ref{eq:S0}), which obeys
			\begin{equation} \label{eq:suloclip}
				\|\bu\|_{W^{1,\infty}} \leq C\|\bf\|_{L^\infty}.
			\end{equation}
		\end{theorem}
		\begin{proof}
			Let $\bf \in L^\infty(\Omega)$. We always have
			\begin{equation*}
				\|\bf\|_{L^2_\uloc} \leq C \|\bf\|_{L^\infty}.
			\end{equation*}
			The Sobolev embeddings in bounded domains adapts into the continuous inclusion
			\begin{equation*}
				H^2_\uloc(\Omega) \hookrightarrow W^{1,q}_\uloc(\Omega), \quad 2 \leq q < \infty.
			\end{equation*}
			Since Theorem \ref{thm:ellreg} ensures the existence of a solution to (\ref{eq:S0}) in $H^2_\uloc(\Omega)$, we also have existence of a solution in $W^{1,q}_\uloc(\Omega)$. Besides, the inclusions of Lebesgue spaces imply
			\begin{equation*}
				W^{1,q}_\uloc(\Omega) \hookrightarrow H^1_\uloc(\Omega), \quad 2\leq q < \infty.
			\end{equation*}
			Hence, the uniqueness of a solution in $H^1_\uloc$, ensured by Theorem \ref{thm:suloc}, implies that there exists at most one solution of (\ref{eq:S0}) in $W^{1,q}_\uloc(\Omega)$. In the end, (\ref{eq:S0}) admits a unique solution $\bu \in W^{1,q}_\uloc(\Omega)$ for $2 \leq q < \infty$, with estimate
			\begin{equation*}
				\|\bu\|_{W^{1,q}_\uloc} \leq C_q \|\bf\|_{L^\infty}.
			\end{equation*}
			Now the method is exactly the same as in Theorem \ref{thm:ellreg} to prove that $\bu$ belongs to $W^{2,q}_\uloc(\Omega)$. To do so, the only extra result we require is Ne\v{c}as inequality in the general $L^q$ framework, see \cite[Ex. III.3.4, p. 175]{galdi}, which provides the very same pressure estimates as for $q=2$. In the end we obtain the well-posedness of the problem in $W^{2,q}_\uloc(\Omega)$, with estimate
			\begin{equation*}
				\|\bu\|_{W^{2,q}_\uloc} \leq C \|\bf\|_{L^\infty}, \quad 2\leq q<\infty.
			\end{equation*}
			Now use that
			\begin{equation*}
				W^{2,4}_\uloc(\Omega) \hookrightarrow W^{1,\infty}(\Omega) \hookrightarrow W^{1,4}_\uloc(\Omega).
			\end{equation*}
			As previously, the first embedding provides existence of a solution $\bu \in W^{1,\infty}(\Omega)$, together with estimate (\ref{eq:suloclip}), and the second one ensures uniqueness.
		\end{proof}
	\subsection{Stability estimate for the transport in the strip} \label{sbsec:stabstrip}
	
	The transport equation (\ref{eq:t}) is still well-posed on $\Omega = \R\times(0,1)$ and Proposition \ref{prop:t} still applies. The lemmas related to the properties of the characteristics are also valid still, up to minor adaptations mentioned when required in the following. The only adaptation demanding particular attention is the stability estimate from Proposition \ref{prop:stabb}, stated as follows.
	
	\begin{proposition} \label{prop:stabubd}
		Let $\bu_i \in L^\infty(\R_+;W^{1,\infty}(\Omega))$ with ${\bu_i}|_{\p\Omega} \equiv \bzero$ and $\div \, \bu_i \equiv 0$, for $i= 1,2$. Let $\rho_0 \in L^\infty(\Omega)$ and set $\rho_i$ the solution of (\ref{eq:t}) associated to $\bu_i$ with initial datum $\rho_0$. There exists $\bar{T}(\|\nabla \bu_i\|_{L^\infty}) > 0$ such that for any $T \in [0,\bar{T}]$,
		\begin{equation*} \|\rho_1-\rho_2\|_{L^\infty(0,T;H^{-1}_\uloc)}  \\\leq BT(1+MT)^{1/2}e^{CT\|\nabla\bu_1\|_{L^\infty(0,T;L^\infty)}}\|\bu_1-\bu_2\|_{L^\infty(0,T;L^2_\uloc)},
		\end{equation*}
		where $B :=  C\|\rho_0\|_{L^\infty}$ and $M:= \max_i \|\bu_i\|_{L^\infty}$.
	\end{proposition}
	
	\begin{proof}
		The goal is to bound the following quantity for any test function $\vphi \in C^\infty_\rc(\Omega)$ uniformly in $k \in \Z$ and with respect to $t \in [0,\bar{T}]$ where $\bar{T}$ is determined further. To apply Liouville theorem gives
		\begin{align*}
			 I_{\vphi,k} &:= \int_\Omega \big( \rho_1(t,x) - \rho_2(t,x)\big) (\chi_k\vphi)(x) \ud x \\
				&= \int_\Omega \rho_0(x) \big( (\chi_k\vphi)(X_1(t,x)) - (\chi_k\vphi)(X_2(t,x))\big) \ud x \\
				&= \int_\Omega \rho_0(x) (X_1(t,x) - X_2(t,x)) \cdot \int_0^1 \nabla(\chi_k\vphi) (X^\theta(t,x)) \ud \theta \ud x,
		\end{align*}
		where $X^\theta(t,x) := \theta X_1(t,x) + (1-\theta) X_2(t,x)$. Since $\rho_i$ are the push-forwards of $\rho_0$ by $X_i$, the respective transports occur at finite speed, bounded by $M$. Since $\chi_k$ is supported in $U_k^*$, the support of $(\chi_k\vphi)\circ X_i(t)$ is included in
		\begin{equation*}
			U_k^{M,t} := \{ k - 1 - Mt < x < k + 2 + Mt \}.
		\end{equation*}
		Hence Hölder's inequality applies as follows
		\begin{equation*}
			|I_{\vphi,k}| \leq \|\rho_0\|_{L^\infty}\|X_1(t)-X_2(t)\|_{L^2(U_k^{M,t})}\int_0^1 \|\nabla (\chi_k\vphi)(X^\theta(t))\|_{L^2} \ud \theta.
		\end{equation*}
		We saw in Proposition \ref{prop:stabb} that there exists $\bar{T}(\|\nabla\bu_i\|_{L^\infty}) > 0$ such that $X^{\theta}(t)$ performs a change of variable with jacobian determinant uniformly bounded with respect to $t \in [0,\bar{T}]$ and $\theta \in [0,1]$, meaning there exists a constant $C>0$ such that
		\begin{equation*}
			\forall t \in [0,\bar{T}], \theta \in [0,1], \qquad	\| \nabla(\chi_k\vphi)(X^\theta(t))\|_{L^2} \leq C \|\nabla(\chi_k\vphi)\|_{L^2} \leq C_\chi \|\vphi\|_{H^1}.
		\end{equation*}
		Besides, Lemma \ref{lem:stabchar} applies on the bounded domain $U_k^{M,t}$, providing
		\begin{equation*}
			\|X_1(t) - X_2(t)\|_{L^2(U_k^{M,t})} \leq te^{Ct\|\nabla\bu_1\|_{L^\infty(0,t;L^\infty)}}\|\bu_1-\bu_2\|_{L^\infty(0,t;L^2(U_k^{M,t}))}.
		\end{equation*}
		From considerations similar to those of Lemma \ref{lem:gluloc} we have
		\begin{equation*}
			\|\bu_1-\bu_2\|_{L^\infty(0,t;L^2(U_k^{M,t}))} \leq C(1+Mt)^{1/2}\|\bu_1-\bu_2\|_{L^\infty(0,t;L^2_\uloc)},
		\end{equation*}
		with finite right hand side, since $\bu_i \in L^\infty(\R_+;W^{1,\infty}(\Omega))$. Combining these last equalities lead to
		\begin{equation*}
		\forall t \in [0,\bar{T}],\quad	|I_{\vphi,k}| \leq C\|\rho_0\|_{L^\infty} t (1+Mt)^{1/2}e^{Ct\|\nabla\bu_1\|_{L^\infty(\R_+;L^\infty)}}\|\bu_1-\bu_2\|_{L^\infty(0,t;L^2_\uloc)} \|\vphi\|_{H^1}.
		\end{equation*}
		Taking the supremum over the test functions, $k \in \Z$ and $t \in [0,T]$, we see that for any $T \in [0,\bar{T}]$ we have
		\begin{equation*} \|\rho_1-\rho_2\|_{L^\infty(0,T;H^{-1}_\uloc)} \leq BT(1+MT)^{1/2}e^{CT\|\nabla\bu_1\|_{L^\infty(0,T;L^\infty)}}\|\bu_1-\bu_2\|_{L^\infty(0,T;L^2_\uloc)}.
		\end{equation*}
	\end{proof}
	
	\subsection{Proof of Theorem \ref{thm:2}} \label{sbsec:proofstrip}
	
		The proof essentially follows the same path as in Theorem \ref{thm:1}. For this reason we recall briefly the similar steps and focus on the parts that differ from this former case.
		
		\paragraph{Local existence : } Set $\rho^0 \equiv \rho_0$ in $L^\infty(\R_+;L^\infty(\Omega))$. Define, thanks to Proposition \ref{prop:t} and Theorem \ref{thm:gwplip}, the following sequences
		\begin{equation*}
		\forall N \in \N, \quad \rho^N \in L^\infty(\R_+;L^\infty(\Omega)), \quad \bu^N \in L^\infty(\R_+,W^{1,\infty}(\Omega)),
		\end{equation*}
		satisfying the partial problems
		\begin{equation} \label{eq:Tlocn}
			\left\{
			\begin{array}{rcll}
			\p_t\rho^{N+1} +\bu^N\cdot\nabla\rho^{N+1} &= &0 &\text{in } \R_+\times\Omega, \\
			\rho^{N+1}(0,\cdot) &= &\rho_0 &\text{in } \Omega,
			\end{array}
			\right.
		\end{equation}
		and
		\begin{equation} \label{eq:Stolocn}
		\left\{
		\begin{array}{rcll}
		-\Delta \bu^N + \nabla p^N &= &-\rho^N\be_z &\text{in } \R_+\times\Omega, \\
		\div \, \bu^N &= &0 &\text{in } \R_+\times\Omega, \\
		{\bu^N} &= &\bzero &\text{in } \R_+\times\p\Omega, \\
		\int u_1^N &= &0 &\text{in } \R_+.\\
		\end{array}
		\right.
		\end{equation}
		The uniforms bounds, with $B:= C\|\rho_0\|_{L^\infty}$, remain true,
		\begin{equation} \label{eq:unesloc}
			\|\rho^N\|_{L^\infty(\R_+,L^\infty)} \leq B, \quad \|\bu^N\|_{L^\infty(\R_+;W^{1,\infty})} \leq B.
		\end{equation}
		Therefore we still have $\mathrm{weak}^*$ convergence of $\rho^N, \bu^N$ and $\nabla \bu^N$, up to the extraction of subsequences. Beside, Proposition \ref{prop:stabubd} ensures the existence of a $\bar{T}(\|\rho_0\|_{L^\infty}) > 0$ such that for any $T \in [0,\bar{T}]$,
		\begin{equation*}
			\|\rho^{N+1}-\rho^N\|_{L^\infty(0,T;H^{-1}_\uloc)} \leq BT(1+BT)^{1/2}e^{BT}\|\rho^N-\rho^{N+1}\|_{L^\infty(0,T;H^{-1}_\uloc)},
		\end{equation*}
		where we have plugged (\ref{eq:unesloc}). Therefore, up to the choice of a small enough $T>0$, $(\rho^N)_N$ is a Cauchy sequence in $L^\infty(0,T;H^{-1}_\uloc(\Omega))$, which implies that $(\bu^N)_N$ is also a Cauchy sequence in $L^\infty(0,T;H^1_\uloc(\Omega))$, with limit denoted $\bu$. The $\mathrm{weak}^*$ convergence of $(\bu^N)_N$ and $(\nabla \bu^N)_N$ ensures that $\bu$ also belongs to $L^\infty(0,T;W^{1,\infty}(\Omega))$. In particular, $\bu^N$ and its derivatives converge in $L^1_\loc([0,T]\times\Omega)$, which, together with the $\mathrm{weak}^*$ convergence of $(\rho^N)_N$, is enough to pass to the limit in the weak formulation of partial problems (\ref{eq:Tlocn}) and (\ref{eq:Stolocn}). We obtain a local in time solution $(\rho,\bu)$ of (\ref{eq:S0}), with regularity
		\begin{equation*}
			 L^\infty(0,T;L^\infty(\Omega))\times L^\infty(0,T;W^{1,\infty}(\Omega)).
		\end{equation*}
		
		\paragraph{Local uniqueness :} Let $(\rho_i,\bu_i)$ be two such solutions of (\ref{eq:S0}). The contraction adapts thanks to Proposition \ref{prop:stabubd} in
		\begin{equation*}
		 \|\rho_1-\rho_2\|_{L^\infty(0,T;H^{-1}_\uloc)} \leq BT(1+BT)^{1/2}e^{BT}\|\rho_1-\rho_2\|_{L^\infty(0,T;H^{-1}_\uloc)},
		\end{equation*}
		which implies uniqueness for $T >0$ small enough. 
		
		\paragraph{Globality :} The extension proves just as in the bounded case, see the proof of Theorem \ref{thm:1}.
		
		\subsection{Stability estimate for the system in the strip} \label{sbsec:stabstripsys}
			The result and the proof are identical to the ones of Proposition \ref{prop:stabsys}, replacing the stability estimate (\ref{eq:st0}) of the bounded case by the one in the strip from Proposition \ref{prop:stabubd}.
			\begin{proposition} \label{prop:stabstrip}
				Let $\rho_{0,i} \in L^\infty(\Omega)$ and $\rho_i$ be the solution of (\ref{eq:st}) with initial datum $\rho_{0,i}$, for $i = 1,2$. There exists $C = C(\Omega, \|\rho_{0,i}\|_{L^\infty}) > 0$ such that
				\begin{equation} \label{eq:stabst}
				\forall T \in \R_+, \quad \|\rho_1-\rho_2\|_{L^\infty(0,T;H^{-1}_\uloc)} \leq Ce^{CT} \|\rho_{0,1}-\rho_{0,2}\|_{H^{-1}_\uloc}.
				\end{equation}
			\end{proposition}

	\appendix
	\section{Appendix}
		\subsection{Proof of Lemma \ref{lem:comp}} \label{ap:comp}
			 Let $m \in \N, k \in \Z, 1 < q < \infty$ and $u \in W^{m,q}_\uloc(\Omega)$. Since $\chi_k = 1$ on $U_k$,
			\[ \|u\|_{W^{m,q}(U_k)} \leq \|\chi_k u\|_{W^{m,q}}. \]
			The support of $\chi_k$ being $U_k^*$, one has
			\[ \|\chi_k u\|_{W^{m,q}} \leq C_{\chi,m} \|u\|_{W^{m,q}(U_k^*)}.\]
			One can split this last norm as follows
			\[ \|u\|_{W^{m,q}(U_k^*)} \leq \sum_{\ell = k-1}^{k+1}\|u\|_{W^{m,q}(U_\ell)}. \]
			These three inequalities prove the first assertion,
			\begin{equation*}
			\sup_{k\in\Z} \|u\|_{W^{m,q}(U_k)} \simeq \sup_{k\in\Z} \|u\|_{W^{m,q}(U_k^*)} \simeq \|u\|_{W^{m,q}_\uloc}.
			\end{equation*}
			
			Let $m \in \N^*$ and $1<q<\infty$. For readability we adopt the following notations in the rest of this proof. For any $u \in W^{-m,q}(U)$ and $\vphi \in W^{m,q'}_0(U)$ where $U$ is a subdomain of $\Omega$, denote
			\begin{equation*}
			\|u\|_U := \|u\|_{W^{-m,q}(U)}, \quad \|u\|_\uloc := \|u\|_{W^{-m,q}_\uloc}, \qquad \|\vphi\|_{0,U} = \|\vphi\|_{W^{m,q'}_0(U)}
			\end{equation*}
			and the duality brackets 
			\begin{equation*}
			\langle u, \vphi \rangle_U = \langle u, \vphi\rangle_{W^{-m,q}(U),W^{m,q'}_0(U)}.
			\end{equation*}
			The inclusion $\{ \vphi : \|\vphi\|_{0,U_k} = 1 \} \subset \{ \vphi : \|\vphi\|_{0,U_k^*} = 1 \}$ provides the first inequality
			\begin{equation*}
			\|u\|_{W^{-m,q}(U_k)} \leq \|u\|_{W^{-m,q}(U_k^*)}.
			\end{equation*}
			Let us show the remaining direct inequality. We use the definition of the dual norm, and recall that $(\chi_k)_k$ is a partition of the unity, up to a factor $2$. Also notice that a product $\chi_\ell \vphi$ with $\smash{\vphi \in W^{m,q'}_0(U_k^*)}$ has possibly non-empty support only if $|\ell-k|\leq 3$. These remarks justify each step of the following computations, for any $k\in\Z$,
			\begin{align*}
			\|u\|_{U_k^*} 
			&= \sup_{\|\vphi\|_{0,U_k^*}=1} \langle u, \vphi\rangle_{U_k^*} \\
			&\simeq \sup_{\|\vphi\|_{0,U_k^*}=1} \sum|_{\ell-k|\leq 3} \langle u, \chi_\ell \vphi \rangle_{U_k^*}\\
			&\leq \sup_{\|\vphi\|_{0,U_k^*}=1} \sum|_{\ell-k|\leq 3} \langle  \chi_\ell u, \vphi \rangle_{U_k^*}\\
			&\leq \sup_{\|\vphi\|_{0,\Omega}=1} \sum|_{\ell-k|\leq 3} \langle  \chi_\ell u, \vphi \rangle_{\Omega};\\
			&\lesssim \|u\|_{\uloc}.
			\end{align*}
			Finally, the reciprocal inequality is proved by noticing that $\chi_k\vphi$ belongs to $W^{m,q'}_0(U_k^*)$ for any $\vphi \in W^{m,q'}_0(\Omega)$;
			\begin{align*}
			\|\chi_k u\|_{\Omega} 
			&= \sup_{\|\vphi\|_{0,\Omega}=1} \langle u, \chi_k \vphi \rangle_\Omega\\ 
			&\leq \sup_{\|\vphi\|_{0,\Omega}=1} \|u\|_{U_k^*} \|\chi_k\vphi\|_{0,U_k^*}\\
			&\leq C(\chi,m) \|u\|_{U_k^*}\sup_{\|\vphi\|_{0,\Omega}=1} \|\vphi\|_{U_k^*} \\
			&\leq C \|u\|_{U_k^*}.
			\end{align*}
		\hfill$\square$
		
		\begin{remark}
			For $m \in \N^*$, we do not have in general
			\begin{equation*}
			\sup_{k\in\Z} \|u\|_{W^{-m,q}(U_k)} \gtrsim \sup_{k\in\Z} \|u\|_{W^{-m,q}(U_k^*)}.
			\end{equation*}
			Indeed, consider the Dirac mass $\delta_{(0,1/2)}$ belonging to $H^{-2}(\Omega)$ and therefore to $H^{-2}_\uloc(\Omega)$, with $\|\delta\|_{H^{-2}_\uloc} > 0$. Nevertheless, for any $k \in \Z$ we have
			\begin{equation*}
			\forall \vphi \in H^1_0(U_k), \quad \langle \delta, \vphi \rangle_{U_k} = 0.
			\end{equation*}
			The reason is that the support of an element of the negative Sobolev spaces can be included in the complementary of $\cup_k U_k$. This does not happen when the considered subdomains family covers the whole domain, as does $\left( U_k^*\right)_k$.
		\end{remark}
	
		\subsection{Proof of Lemma \ref{lem:gluloc}} \label{ap:gluloc}
			The case $f \in L^2_\uloc(\Omega_n)$ is straightforward,
			\begin{equation*}
				\|f\|_{L^2(\Omega_n)}^2 = \sum_{k=-n}^{n-1} \|f\|_{L^2(U_k)}^2 \leq 2n\|f\|_{L^2_\uloc}^2.
			\end{equation*}
			The case $f \in H^{-1}_\uloc(\Omega_n)$ requires a little more care. We use notations from the proof of Lemma \ref{lem:comp}. Notice that $\sum_{\ell=-n-1}^n \chi_\ell = 2$ on $\Omega_n$, so for any $\vphi \in H^1_0(\Omega_n)$ we have
			\begin{align*}
				\langle f, \vphi \rangle_{\Omega_n} 
					&\simeq \sum_{\ell = -n-1}^n \langle f, \chi_\ell \vphi \rangle_{U_\ell^*} \\
					&\lesssim \sum_{\ell = -n-1}^n \|f\|_{H^{-1}(U_\ell^*)}\|\chi_\ell\vphi\|_{H^1} \\
					&\lesssim \|f\|_{H^{-1}_\uloc} \sum_{\ell =-n-1}^n \|\chi_\ell \vphi\|_{H^1}\\
					&\lesssim \|f\|_{H^{-1}_\uloc}\sum_{\ell=-n-1}^n \|\vphi\|_{H^1(U_\ell^*)} \\
					&\lesssim \|f\|_{H^{-1}_\uloc} (2n+2)^{1/2} \left( \sum_{\ell=-n-1}^n \|\vphi\|_{H^1(U_\ell^*)}^2 \right)^{1/2},
			\end{align*}
			
			\noindent where we used Lemma \ref{lem:comp}. Now, bound $2n+2$ by $4n$ and notice that the last sum is equivalent to $\|\vphi\|_{H^1}$ to complete the proof. \hfill $\square$
		
	\subsection{Proof of Lemma \ref{lem:tec}} \label{ap:rel}
		Set $(E_{n,k})_{n,k}$ a family of positive real numbers, indexed by the couples $(n,k) \in \N^2$ such that $1 \leq k \leq n$, non-decreasing according to $k$ for a fixed $n$, obeying
		\begin{equation} \label{eq:obey}
			\forall 1 \leq k \leq n, \quad E_{n,k} \leq C\big(E_{n,k+1} - E_{n,k} + F(k+1)\big)
		\end{equation}
		and satisfying
		\begin{equation} \label{eq:init} 
			\forall n, \quad E_{n,n} \leq C F n,
		\end{equation}
		where $F$ is a constant playing the role of $\|\bf\|_{H^{-1}_\uloc}^2$. Let us show that there exists $\alpha > 0$ and $k_0 \in \N^*$ such that
		\begin{equation} \label{eq:rec}
		\forall k,n, \quad k_0 \leq k \leq n \quad \implies \quad E_k \leq \alpha F k.
		\end{equation}
		By (\ref{eq:init}) we already know that (\ref{eq:rec}) is satisfied for any $n \in \N^*$ and $k=n$, with $\alpha = C$. For a fixed $n$, let $k_0$ be the greatest index such that (\ref{eq:rec}) is not satisfied, meaning
		\begin{equation} \label{eq:contr}
			E_{n,k_0} > \alpha F k_0.
		\end{equation}
		Therefore, plugging (\ref{eq:contr}) in (\ref{eq:obey}) and using the definition of $k_0$ provides
		\begin{equation*}
			\alpha(1+C) F k_0 \leq CF(\alpha + 1)(k_0 + 1),
		\end{equation*}
		which is equivalent to
		\begin{equation*}
			\frac{k_0}{k_0+1} \leq \frac{C}{C+1}\left( 1 + \frac{1}{\alpha} \right) =: C_\alpha.
		\end{equation*}
		Up to the choice of a greater $\alpha$, we can assume that $C_\alpha < 1$. This implies that $k_0 \leq \frac{C_\alpha}{1-C_\alpha}$, independently of $n$ and $F$. Therefore, we conclude that for any $k,n$ such that $k_0\leq k\leq n$ we have
		\begin{equation*}
			E_k \leq \alpha F k.
		\end{equation*}\hfill $\square$
	
	\section*{Acknowledgement}
		The author thanks Anne-Laure Dalibard and Julien Guillod for introducing him to this problem, for their guidance and the quality of their supervision. The author also thanks Amina Mecherbet for giving him early access to one of her production, and for her useful comments about the present work.
	
	\subsection*{Fundings}
		This project has received funding from the European Research Council (ERC) under the European Union's Horizon 2020 research and innovation program Grant agreement No 637653, project BLOC ``Mathematical Study of Boundary Layers in Oceanic Motion''. This work was also supported by the SingFlows project, grant ANR-18-CE40-0027 of the French National Research Agency (ANR).

	\subsection*{Conflict of interest}
		The author declares that there is no conflict of interest regarding the publication of this article.

\bibliographystyle{plain}
\bibliography{stokestransport}

\begin{thebibliography}{10}

\bibitem{Alazard_2016}
T.~Alazard, N.~Burq, and C.~Zuily.
\newblock Cauchy theory for the gravity water waves system with non-localized
  initial data.
\newblock {\em Annales de l{\textquotesingle}Institut Henri Poincare (C) Non
  Linear Analysis}, 33(2):337--395, mar 2016.

\bibitem{bae}
Hantaek Bae and Rafael Granero-Belinch{\'o}n.
\newblock Global existence for some transport equations with nonlocal velocity.
\newblock {\em Advances in Mathematics}, 269:197--219, 2015.

\bibitem{boyer2012mathematical}
Franck Boyer and Pierre Fabrie.
\newblock {\em Mathematical Tools for the Study of the Incompressible
  Navier-Stokes Equations andRelated Models}, volume 183.
\newblock Springer Science \& Business Media, 2012.

\bibitem{castro2019global}
{\'A}ngel Castro, Diego C{\'o}rdoba, and Daniel Lear.
\newblock Global existence of quasi-stratified solutions for the confined {IPM}
  equation.
\newblock {\em Archive for Rational Mechanics and Analysis}, 232(1):437--471,
  2019.

\bibitem{cordoba2011lack}
Diego Cordoba, Daniel Faraco, and Francisco Gancedo.
\newblock Lack of uniqueness for weak solutions of the incompressible porous
  media equation.
\newblock {\em Archive for rational mechanics and analysis}, 200(3):725--746,
  2011.

\bibitem{dalibard}
Anne-Laure Dalibard and Christophe Prange.
\newblock Well-posedness of the stokes--coriolis system in the half-space over
  a rough surface.
\newblock {\em Analysis \& PDE}, 7(6):1253--1315, 2014.

\bibitem{galdi}
Giovanni Galdi.
\newblock {\em An introduction to the mathematical theory of the
  {N}avier-{S}tokes equations: Steady-state problems}.
\newblock Springer Science \& Business Media, 2011.

\bibitem{gerard2010relevance}
David G{\'e}rard-Varet and Nader Masmoudi.
\newblock Relevance of the slip condition for fluid flows near an irregular
  boundary.
\newblock {\em Communications in Mathematical Physics}, 295(1):99--137, 2010.

\bibitem{hofer}
Richard~M H{\"o}fer.
\newblock Sedimentation of inertialess particles in {S}tokes flows.
\newblock {\em Communications in Mathematical Physics}, 360(1):55--101, 2018.

\bibitem{kiselev2021small}
Alexander Kiselev and Yao Yao.
\newblock Small scale formations in the incompressible porous media equation.
\newblock {\em arXiv preprint arXiv:2102.05213}, 2021.

\bibitem{Ladyzhenskaya_1983}
O.~A. Ladyzhenskaya and V.~A. Solonnikov.
\newblock Determination of the solutions of boundary value problems for
  stationary stokes and {N}avier-{S}tokes equations having an unbounded
  {D}irichlet integral.
\newblock {\em Journal of Soviet Mathematics}, 21(5):728--761, mar 1983.

\bibitem{mecherbetsed}
Amina Mecherbet.
\newblock Sedimentation of particles in {S}tokes flow.
\newblock {\em arXiv preprint arXiv:1806.07795}, 2018.

\bibitem{mecherbet}
Amina Mecherbet.
\newblock On the sedimentation of a droplet in {S}tokes flow, 2020.

\bibitem{santambrogio}
Filippo Santambrogio.
\newblock Optimal transport for applied mathematicians.
\newblock {\em Birk{\"a}user, NY}, 55(58-63):94, 2015.

\bibitem{szekelyhidi2012relaxation}
L{\'a}szl{\'o} Sz{\'e}kelyhidi~Jr.
\newblock Relaxation of the incompressible porous media equation.
\newblock In {\em Annales scientifiques de l'Ecole normale sup{\'e}rieure},
  volume~45, pages 491--509, 2012.

\bibitem{temam2001navier}
Roger Temam.
\newblock {\em {N}avier-{S}tokes equations: theory and numerical analysis},
  volume 343.
\newblock American Mathematical Soc., 2001.

\end{thebibliography}
\end{document}